\documentclass[12pt, reqno]{amsart}
\usepackage{amssymb,amsmath,amsthm,mathrsfs,verbatim}
\usepackage[all, knot]{xy}
\usepackage{rotating}
\usepackage{cancel}
\usepackage{mathtools}
\usepackage[OT2,T1]{fontenc}
\newtheorem{theorem}{Theorem}
\newtheorem{lemma}[theorem]{Lemma}

\newtheorem{proposition}[theorem]{Proposition}
\newtheorem{definition}[theorem]{Definition}

\newcommand{\RR}{s}
\theoremstyle{remark}

\newtheorem*{remark}{Remark}

\newtheorem{example}{Example}
\numberwithin{theorem}{section} \numberwithin{equation}{section}
\numberwithin{figure}{section}
\numberwithin{example}{section}

\newcommand{\wt}{k}
\newcommand{\wtabs}{\kappa}
\newcommand{\FF}{\mathcal{F}}
\newcommand{\FB}{\mathbb{F}}

\newcommand{\QQ}{\mathcal{Q}}

\newcommand{\MPexcwt}[2]{\mathcal{M}_{#1}^{\left(#2\right)}}
\newcommand{\MPexc}[1]{\mathcal{M}_{\wt}^{\left(#1\right)}}
\newcommand{\MPwt}[2]{\mathcal{M}_{#1,#2}}
\newcommand{\MP}[1]{\mathcal{M}_{\wt,#1}}

\newcommand{\MPt}[1]{\MPexc{\Ek_D}}

\newcommand{\C}{\mathbb{C}}
\newcommand{\CC}{\mathcal{C}}

\newcommand{\Z}{\mathbb{Z}}
\newcommand{\N}{\mathbb{N}}

\newcommand{\SL}{\operatorname{SL}}

\newcommand{\Ctau}{\mathscr{C}}

\newcommand{\re}{\textnormal{Re}}
\newcommand{\im}{\textnormal{Im}}

\def\H{\mathbb{H}}

\newcommand{\PP}{\mathcal{P}}
\newcommand{\Ek}{\mathcal{E}}

\newcommand{\Mfun}{\PP}

\begin{document}
\title[Modular local polynomials]{Modular local polynomials}

\author{Kathrin Bringmann} 
\address{Mathematical Institute\\University of
Cologne\\ Weyertal 86-90 \\ 50931 Cologne \\Germany}
\email{kbringma@math.uni-koeln.de}
\author{Ben Kane}
\address{Department of Mathematics\\ University of Hong Kong\\ Pokfulam\\ Hong Kong}
\email{bkane@maths.hku.hk}
\date{\today}
\thanks{The research of the first author was supported by the Alfried Krupp Prize for Young University Teachers of the Krupp Foundation and by the Deutsche Forschungsgemeinschaft (DFG) Grant No. BR 4082/3-1.  Most of this research was conducted while the second author was a postdoc at the University of Cologne.
}

\begin{abstract}
In this paper, we consider modular local polynomials.  These functions satisfy modularity while they are locally defined as polynomials outside of an exceptional set.  We prove an inequality for the dimension of the space of such forms when the exceptional set is given by certain natural geodesics related to binary quadratic forms of (positive) discriminant $D$.  We furthermore show that the dimension is the largest possible if and only if $D$ is an even square.  Following this, we describe how to use the methods developped in this paper to establish an algorithm which explicitly determines the space of modular local polynomials for each $D$.
\end{abstract}

\maketitle
\section{Introduction}

The most elementary non-constant modular forms are given by the Eisenstein series.  These functions are both holomorphic and their modularity follows by an elementary argument.  
For very simple subgroups of $\SL_2(\Z)$, one can construct polynomials which are modular forms; for example, $\tau^2+1$ is a weight $-2$ modular form for $<S>$, where $S:=\left(\begin{smallmatrix}0&-1\\ 1&0\end{smallmatrix}\right)$.  Subgroups for which some polynomials are modular are rather special though, and Knopp \cite{Knopp} showed that if all even weight modular forms of a given subgroup are polynomials and all odd weight modular forms are square roots of polynomials, then the group is nondiscontinuous.
It is easy to check that the only polynomial functions which are modular on all of $\SL_2(\Z)$ are constant functions, so looking in this direction for simpler modular forms than the Eisenstein series seems fruitless.  One could however ask whether a more elementary modular form on $\SL_2(\Z)$ could be constructed.  Functions along this vein with connections to modular forms have been previously considered.  Zagier \cite{ZagierRQ} constructed cusp forms $f_{k,D}$ from rational functions related to quadratic forms of discriminant $D>0$ which turned out to play a major role in understanding the Shimura and Shintani lifts between integral and half-integral weight modular forms \cite{KohnenZagier}.  Zagier \cite{ZagierQuantum} further constructed functions $F_{k,D}$ on the real line which are built out of polynomials and are closely related to $f_{k,D}$ and their periods.  

In this paper, we investigate functions which are local polynomials up to a prescribed exception set along which they need not be differentiable or even continuous. We call local polynomials which satisfy modularity and whose value along the exceptional set is an average of the values in adjacent connected components \begin{it}modular local polynomials.\end{it}  We specifically choose a certain nowhere dense exceptional set $E_D$ (which for $D>0$ is defined in \eqref{eqn:EDdef}).
  The motivation for this particular choice of exceptional set comes from certain examples of locally harmonic Maass forms investigated in \cite{BKW}.  Roughly speaking, these are functions which satisfy weight $\wt$ modularity and are almost everywhere annihilated by the weight $\wt$ hyperbolic Laplace operator $\Delta_{\wt}$.  In the special case that the locally harmonic Maass forms from \cite{BKW} are locally holomorphic, the authors and Kohnen proved that they are modular local polynomials.  
Moreover, the locally harmonic Maass forms in \cite{BKW} are intimately related to Zagier's cusp forms $f_{k,D}$ 
via a natural operator in the theory of harmonic weak Maass forms.  Through this connection, one can determine that the $D$th locally harmonic Maass form defined in \cite{BKW} is a modular local polynomial precisely when $f_{k,D}$ vanishes, motivating the study of modular local polynomials here.

We denote the space of (even integral) weight $\wt\leq 0$ modular local polynomials with exceptional set $E$ by $\MPexcwt{\wt}{E}$.  Noting that the degree of the polynomial in each connected component is at most $\wtabs:=|\wt|$, one easily finds that each $\MPexcwt{\wt}{E}$ is a finite-dimensional vector space, and it is hence interesting to investigate their dimensions.  Focusing on the exceptional sets arising in the first treatment of locally harmonic Maass forms \cite{BKW}, we abbreviate $\MP{D}:=\MPexcwt{\wt}{E_D}$.  Our main result compares the dimension of $\MP{D}$ with the dimension $\wtabs+1$ of polynomials of degree at most $\wtabs$ and the (minimal) number $r_{\FB}$ of connected components which have non-trivial intersection with a given fundamental domain $\FB$.  
\begin{theorem}\label{thm:dim}
If $\FB$ is a fundamental domain for $\SL_2(\Z)$ and $\wt<0$, then
$$
\dim\left(\MP{D}\right)\leq \left(\wtabs+1\right)r_{\FB}.
$$
Moreover,
$$
\dim\left(\MP{D}\right) = \left(\wtabs+1\right)\min_{\FB} r_{\FB}
$$
if and only if $D$ is an even square.
\end{theorem}
\begin{remark}
In the case that $\dim\left(\MP{D}\right)$ is strictly smaller than $(\wtabs+1)\min_{\FB}r_{\FB}$, the proof of Theorem \ref{thm:dim} yields specific relations between the polynomials allowed in each connected component.  This gives an algorithm to determine all possible modular local polynomials.  
\end{remark}
If $\wt=0$, the modular local polynomials are local constants.  In this case, the dimension is always maximal, independent of $D$.
\begin{theorem}\label{thm:dim0}
For every $D>0$, we have
$$
\dim\left(\MPwt{0}{D}\right) = \min_{\FB} r_{\FB}.
$$
\end{theorem}
\rm

The methods in this paper are constructive, and they produce an explicit description of the modular local polynomials.
\begin{example}\label{ex:5}
For $D=5$ and $\wt=-2$, the resulting algorithm yields 
$$
\dim\left(\MPwt{-2}{5}\right)=2.
$$
In order to give an explicit description of all weight $-2$ modular local polynomials for $E_5$, we denote the connected component of $\H\setminus E_5$ containing all $\tau\in \H$ with $\im(\tau)>\frac{\sqrt{5}}{2}$ by $\CC_{\infty}$ and the connected component containing $\rho:=e^{\frac{2\pi i}{3}}$ by $\CC_{\rho}$.  
One can show that the value of $\PP\in \MPwt{-2}{5}$ may be uniquely determined by its value in these two connected components.  Running through all $\alpha,\beta\in \C$, the elements of $\MPwt{-2}{5}$ are then given by the 
$\PP_{\alpha,\beta}\in \MPwt{-2}{5}$ satisfying
$$
\PP_{\alpha,\beta}(\tau):=
\begin{cases} 
\alpha &\text{if }\tau\in \CC_{\infty},\\
\beta\left(\tau^2-\tau+1\right)&\text{if }\tau\in \CC_{\rho}.
\end{cases}
$$
The details are worked out in Section \ref{sec:ex}.
\end{example}

The paper is organized as follows.  We first give the formal definitions of 
modular local polynomials in Section \ref{sec:Maass}.  We prove Theorem \ref{thm:dim} in Section \ref{sec:thmdim} 
and Theorem \ref{thm:dim0} in Section \ref{sec:thmdim0}. 
Finally, in Section \ref{sec:ex} we demonstrate how to determine the possible modular local polynomials by explicitly computing a special case.  
\section{Modular local polynomials}\label{sec:Maass}
The main goal of this section is to formally define modular local polynomials.  
For an $\SL_2(\Z)$-invariant nowhere dense 
(see page 42 of \cite{Rudin} for a definition)
 set $E$, a function $\PP:\H\to\C$ is called a \begin{it}local polynomial with exceptional set $E$\end{it} if for every connected component $\CC\subset \H\setminus E$ there exists a polynomial $P_{\CC}$ which satisfies for every $\tau\in \CC$
$$
\PP(\tau)=P_{\CC}(\tau).
$$
Note that this definition does not restrict the value of $\PP$ along the exceptional set $E$, since $\PP$ may exhibit discontinuities.  However, it is 
natural to relate the value of a function $\Mfun:\H\to \C$ for $\tau\in \H$ with the average of its value on the connected components
$$
\Ctau_{\tau}=\Ctau_{\tau}^{(E)}:=\left\{ \CC\subset\H\setminus E\Big| \tau\in \CC\cup \partial\CC\right\}
$$
containing $\tau$ in their closure.  In the case of existence, we hence define
\begin{equation}\label{eqn:avgdef}
\mathcal{A}(\Mfun)(\tau)=\mathcal{A}_{E}(\Mfun)(\tau):=\frac{1}{\#\Ctau_{\tau}}\sum_{\CC\in \Ctau_{\tau}}\lim_{\substack{w\in \CC\\ w\to \tau}} \Mfun(w).
\end{equation}
\begin{definition}
For 
an $\SL_2(\Z)$-invariant 
 nowhere dense set $E$, a function $\Mfun:\H\to\C$ is called a weight $\wt$ \begin{it}modular local polynomial\end{it} with exceptional set $E$ if $\Mfun$ satisfies the following conditions:
\noindent
\begin{enumerate}
\item
For every $\gamma\in \SL_2(\Z)$, one has $\Mfun|_{\wt}\gamma = \Mfun$.
\item
The function $\Mfun$ is a local polynomial with exceptional set $E$.
\item
The limit defining $\mathcal{A}\left(\Mfun\right)$ in \eqref{eqn:avgdef} exists for all $\tau\in \H$, and 
\begin{equation}\label{eqn:Favg}
\Mfun=\mathcal{A}\left(\Mfun\right)
\end{equation}
\end{enumerate}
\end{definition}
As mentioned in the introduction, modular local polynomials are special cases of functions called locally harmonic Maass forms which were introduced in \cite{BKW}.  
The locally harmonic Maass forms in \cite{BKW} exhibit discontinuities on the exceptional sets
\begin{equation}\label{eqn:EDdef}
E_D:=\bigcup_{Q\in \QQ_D} S_Q,
\end{equation}
where $\QQ_D$ denotes the binary quadratic forms of discriminant $D$ and for $Q=[a,b,c]\in\QQ_D$, we define 
$$
S_Q:=\left\{ \tau\in \H\Big|\  a|\tau|^2+b\re(\tau)+c=0\right\}.
$$
The definition of locally harmonic Maass forms given in \cite{BKW} had an altered (inequivalent) version of condition \eqref{eqn:Favg}, 
but one can show that the locally harmonic Maass forms in \cite{BKW} also satisfy \eqref{eqn:Favg}.  The condition given in \cite{BKW} implies an extra symmetry of the possible polynomials at points of intersection between the boundaries of more than two connected components.  Condition (3) is natural because (as we see later) modularity in $\H\setminus E$ implies that \eqref{eqn:Favg} is satisfied for $\tau\in \H$ if and only if it is satisfied for $\gamma\tau$ for all $\gamma\in \SL_2(\Z)$.

\section{Proof of Theorem \ref{thm:dim}}\label{sec:thmdim}
We break the proof of Theorem \ref{thm:dim} into three parts, the upper bound for all $D$, the equality for even square $D$, and finally we show that the equality does not hold when $D$ is not an even square.  We first prove a slightly more general version of the upper bound for arbitrary exceptional sets $E$ in terms of the (minimal) number $r_{\FB}=r_{\FB}^{(E)}$ of connected components covering a fundamental domain $\FB$.  Here a \begin{it}fundamental domain\end{it} is any set which contains precisely one element in $\H$ of each orbit under the action of $\SL_2(\Z)$.  We omit the exceptional set in this notation whenever it is clear from the context.

\begin{lemma}\label{lem:trivdim}
If $\FB$ is a fundamental domain for $\SL_2(\Z)$ 
and $\wt\leq 0$, then
$$
\dim\left(\MPexc{E}\right)\leq \left(\wtabs+1\right)r_{\FB}.
$$
\end{lemma}
\begin{remark}
For the exceptional sets $E_D$ concentrated on in this paper, $r_{\FB}$ is finite because there are only finitely many connected components in a sufficiently small neighborhood of the cusp of $\FB$ (precisely $1$ if $D$ is not a square and $\sqrt{D}$ if $D$ is a square) and any bounded set is covered by finitely many connected components because $E_D$ is nowhere dense.
\end{remark}
\begin{proof}
If $\PP\in \MPexc{E}$, then for each of the connected components $\CC\subset\H\setminus E$, there is a polynomial $P=P_{\PP,\CC}$ such that for all $\tau\in\CC$ one has
$$
\PP(\tau)=P(\tau).
$$
If $r_{\FB}=\infty$, then the inequality in the lemma holds trivially, so we may suppose that there are exactly $r$ connected components $\CC_1,\dots,\CC_r$ which intersect $\FB$ and for each $\PP\in \MPexc{E}$ denote $P_n:=P_{\PP,\CC_n}$.  

Modularity implies that the degree of $P_n$ is at most $\wtabs$, since otherwise the function would not be polynomial in some connected component.  Furthermore, \eqref{eqn:Favg} uniquely determines the value of $\PP$ on $\H$ once its value is known on $\H\setminus E$, while one obtains the value of $\PP$ on $\H\setminus E$ from the value in each $\CC_n$ via modularity.  This yields the inequality.
\end{proof}

We next prove that equality holds  in Theorem \ref{thm:dim} whenever $D$ is an even square.  In our proof, we make use of the standard fundamental domain for $\SL_2(\Z)$, defined by
\begin{multline*}
\FF:=\left\{ \tau\in \H\Big|  -\frac{1}{2}<\re\left(\tau\right)\leq \frac{1}{2}\text{ and }|\tau|>1 \right\}\\
\cup \left\{ \tau\in \H\Big| 0\leq \re\left(\tau\right)\leq \frac{1}{2}\text{ and }|\tau|=1\right\}.
\end{multline*}
In particular, we prove that every possible polynomial of degree at most $\wtabs$ is allowed in each connected component $\CC_1,\dots, \CC_{r}$ (for some $r\in \N$) of 
$$
\FF\setminus\left(E_D\cap \FF\right).
$$
We first construct a local polynomial with a (possibly) slightly larger exceptional set than $E_D$ by adding the image under every element of $\SL_2(\Z)$ of the boundary of the fundamental domain, explicitly given by
\begin{multline*}
\partial \FF=\left\{ \frac{1}{2}+it\Big| t\geq \frac{\sqrt{3}}{2}\right\}\cup \left\{ -\frac{1}{2}+it\Big| t\geq \frac{\sqrt{3}}{2}\right\}\\
 \cup \left\{ \tau\in \H\Big| -\frac{1}{2}\leq \re(\tau)\leq\frac{1}{2}\text{ and }|\tau|=1\right\}.
\end{multline*}
One easily sees that 
$$
\Ek_D:=E_D\cup \bigcup_{\gamma\in \SL_2(\Z)} \gamma \left(\partial \FF\right)
$$
is nowhere dense, and hence it makes sense to consider local polynomials with exceptional set $\Ek_D$.  To ease notation, we denote $\RR_{\FB}:=r_{\FB}^{\left(\Ek_D\right)}$.

\begin{proposition}\label{prop:modularP}
\noindent

\noindent
\begin{enumerate}
\item
For every $D>0$ and $\wt\leq 0$, one has 
$$
\dim\left(\MPt{D}\right)=\left(\wtabs+1\right)\RR_{\FF}=\left(\wtabs+1\right) \min_{\FB} \RR_{\FB}.
$$
\item
If $D$ is an even square and $\wt\leq 0$,
 then 
\begin{equation}\label{eqn:dimmax}
\dim\left(\MP{D}\right)=\left(\wtabs+1\right)r_{\FF}=\left(\wtabs+1\right)\min_{\FB} r_{\FB}.
\end{equation}
\end{enumerate}
\end{proposition}
\begin{remark}
By Lemma \ref{lem:trivdim}, the dimensions in Proposition \ref{prop:modularP} are the largest possible.
\end{remark}
\begin{proof}

\noindent
(1)  We abbreviate $\RR:=\RR_{\FF}$ and denote the connected components intersecting $\FF$ by $\CC_1,\dots,\CC_{\RR}$.  To establish the first identity in part (1), we give an explicit bijection between $\RR$-tuples $\left(P_1,\dots, P_{\RR}\right)$ of polynomials of degree at most $\wtabs$ and weight $\wt$ modular local polynomials $\PP$ with exceptional set $\Ek_D$.  Specifically, we define $\PP$ such that for every $\tau\in \CC_n$ one has
$$
\PP(\tau)=P_n(\tau)
$$
and then (uniquely) specify the value of $\PP$ elsewhere so that modularity and \eqref{eqn:Favg} are satisfied.

Since $\PP$ is modular for $\tau\notin\Ek_D$ and \eqref{eqn:Favg} is satisfied by definition, it remains to prove modularity for $\tau\in \Ek_D$ and that $\PP$ is a local polynomial.  Using $\Ctau_{\gamma \tau} = \gamma\left(\Ctau_{\tau}\right)$, we obtain that
\begin{align*}
\PP\big|_{\wt}\gamma(\tau)& = \frac{\left(c\tau+d\right)^{\wtabs}}{\#\Ctau_{\gamma\tau}}\sum_{\CC\in \Ctau_{\gamma\tau}}\lim_{\substack{w\in \CC \\ w\to \gamma\tau}} \PP\left(w\right)\\
&=\frac{\left(c\tau+d\right)^{\wtabs}}{\#\Ctau_{\tau}}\sum_{\CC\in \Ctau_{\tau}}\lim_{\substack{w\in \CC \\ w\to \tau}} \PP\left(\gamma w\right)\\
&= \frac{1}{\#\Ctau_{\tau}}\sum_{\CC\in \Ctau_{\tau}}\lim_{\substack{w\in \CC \\ w\to \tau}} \left(\frac{c\tau+d}{cw+d}\right)^{\wtabs}\PP(w)\\
&= \frac{1}{\#\Ctau_{\tau}}\sum_{\CC\in \Ctau_{\tau}}\lim_{\substack{w\in \CC \\ w\to \tau}} \PP(w)=\PP(\tau).
\end{align*} 
Finally, $\PP$ is a local polynomial by definition since every polynomial $
P_n
$
has degree at most $\wtabs$, so that $
P_n
|_{\wt}\gamma$ ($\gamma\in \SL_2(\Z)$) is again a polynomial.  Therefore, $\PP$ is a modular local polynomial.  

The upper bound in Lemma \ref{lem:trivdim} implies that we have constructed all such modular local polynomials, establishing the correspondence and hence the first equality for the dimension.

It remains to show that $\RR=\RR_{\FF}$ is minimal among all $\RR_{\FB}$.  For this, suppose that $\FB$ is a fundamental domain for which  $\widetilde{\RR}:=\RR_{\FB}$ is minimal.  To conclude that $\RR=\widetilde{\RR}$, it suffices to prove that 
\begin{equation}\label{eqn:FFdecomp}
\FB\left\backslash\left(\Ek_D\cap\FB\right)\right. =\bigcup_{n=1}^{\RR} \gamma_n\left(\CC_n\right),
\end{equation}
since $\partial \FF\subset\Ek_D$ implies that the sets $\gamma_n\left(\CC_n\right)$ are disjoint.  Equation \eqref{eqn:FFdecomp} follows once we establish that for each $1\leq n\leq \RR$ there exists $\gamma_n\in \SL_2(\Z)$ such that
$$
S_n:=\left\{ \tau\in \FB\Big| \exists \gamma\in \SL_2(\Z)\text{ such that }\tau\in \gamma\left(\CC_n\right)\right\}= \gamma_n\left(\CC_n\right).
$$
Since $S_n$ is non-empty, we may choose $\gamma_n\in \SL_2(\Z)$ such that $S_n\cap \gamma_n\left(\CC_n\right)\neq\emptyset$.  By replacing $S_n$ with the single connected component $\gamma_n\left(\CC_n\right)$, we define a new fundamental domain 
\rm
$$
\FB':=\left(\FB\setminus S_n\right)\cup \gamma_n\left(\CC_n\right)
$$
which satisfies $\RR_{\FB'}\leq \widetilde{\RR}$.  Moreover, $S_n\subseteq \gamma_n\left(\CC_n\right)$, since otherwise $\RR_{\FB'}< \widetilde{\RR}$, contradicting the minimality of $\widetilde{\RR}$.  Note that since $\FB$ is a fundamental domain, every $\tau\in \gamma_n\left(\CC_n\right)$ corresponds to an $\SL_2(\Z)$-equivalent point $\tau_0\in S_n$, and hence $S_n\subseteq\gamma_n\left(\CC_n\right)$ implies that $S_n=\gamma_n\left(\CC_n\right)$.
This finishes the claim of (1).

\noindent
(2)  For the second statement, it remains to prove that $\Ek_D=E_D$ if $D=4m^2$ ($m\in \N_0$).  For this, it suffices to show that $\partial\FF\cap \FF\subset E_D$.  To do so, we define the quadratic forms
\begin{align*}
Q_1&:=[m,0,-m],\\
Q_2&:=[0,2m,-m]
\end{align*}
with discriminant $4m^2=D$.  By a direct computation, we have
\begin{align*}
S_{Q_1}&= \left\{\tau\in \H\Big|  m|\tau|-m=0\right\}=\left\{\tau\in \H\Big|  |\tau|=1\right\},\\
S_{Q_2}&= \left\{\tau\in \H\Big| 2m\re(\tau)-m=0\right\}=\left\{\tau\in \H\Big|  \re\left(\tau\right)=\frac{1}{2}\right\}.
\end{align*}
Therefore 
$$
\partial\FF\cap\FF\subset S_{Q_1}\cup S_{Q_2},
$$
and it follows that $\PP$ is a local polynomial with exceptional set $E_D$.
\end{proof}
As noted above, 
Proposition \ref{prop:modularP} is optimal in the sense that the dimension is the largest possible.  Theorem \ref{thm:dim} implies that Proposition \ref{prop:modularP} is also optimal in the sense that the restriction on $D$ is necessary 
if $\wt<0$.
\begin{proof}[Proof of Theorem \ref{thm:dim}]
By Lemma \ref{lem:trivdim} and Proposition \ref{prop:modularP}, it remains to show that if $D$ not an even square, then for every fundamental domain $\FB$ the dimension of $\MP{D}$ is strictly less than $(\wtabs+1)r_{\FB}$.  For this it suffices to prove that there is a restriction on the polynomials allowed in one of the connected components covering $\FB$.  By modularity, it is enough to determine a restriction for the allowed polynomials in one of the connected components covering $\FF$.  

If $D$ is not a square, then there exists a connected component containing those $\tau\in \H$ with $\im(\tau)>\frac{\sqrt{D}}{2}$ \cite{BKW}.  When restricted to these $\tau\in \H$, every modular local polynomial is a translation invariant polynomial and hence is constant in this connected component.  Since $\wtabs>0$, this yields a restriction on the polynomials allowed in this connected component.

If $D$ is an odd square, then $\rho=e^{\frac{2\pi i}{3}}=\frac{1}{2}+\frac{\sqrt{3}}{2}\notin E_D$, since
$$
a+\frac{1}{2}b+c=a|
\rho
|^2+b
\re(\rho)
+c=0
$$
implies that $b$ is even, and hence also $D\equiv b^2\pmod{4a}$ is even, 
contradicting
 the assumption on $D$.  Thus for every $\PP\in \MP{D}$, there exists a polynomial $P_{\rho}$ for which $\PP=P_{\rho}$ in the connected component $\CC_{\rho}$ including $\rho$.  Furthermore, there is a polynomial $P_{\rho-1}$ associated to the connected component containing $\rho-1$.  Since $\PP$ is translation invariant, we have for $\tau\in \CC_{\rho}$
$$
P_{\rho}(\tau)=\PP(\tau)=\PP(\tau-1)=P_{\rho-1}\left(\tau-1\right).
$$
From this one obtains the relation
\begin{equation}\label{eqn:Ptrans}
P_{\rho-1}(X)=P_{\rho}(X+1)
\end{equation} 
via analytic continuation.  Furthermore, since $-\rho^{-1}=\rho-1$, modularity and the evaluation \eqref{eqn:Ptrans} of $P_{\rho-1}$ yield for $\tau\in \CC_{\rho}$ that
$$
P_{\rho}\left(\tau\right) = \PP(\tau)=\tau^{\wtabs}\PP\left(-\frac{1}{\tau}\right)=\tau^{\wtabs}P_{\rho-1}\left(-\frac{1}{\tau}\right) = \tau^{\wtabs}P_{\rho}\left(-\frac{1}{\tau}+1\right).
$$
By analytic continuation, one has 
\begin{equation}\label{eqn:Prel}
P_{\rho}(X)=X^{\wtabs}P_{\rho}\left(1-\frac{1}{X}\right).
\end{equation}
Since $\wtabs>0$, the dimension of the space of polynomials satisfying \eqref{eqn:Prel} is less than $\wtabs+1$ (for example, the polynomial $X$ does not satisfy \eqref{eqn:Prel}).  
\end{proof}
\section{Proof of Theorem \ref{thm:dim0}}\label{sec:thmdim0}

In this section we consider the case $\wt=0$ (i.e., the modular local polynomials are modular local constants).  
\begin{proof}[Proof of Theorem \ref{thm:dim0}]
By Lemma \ref{lem:trivdim}, we have the upper bound 
\begin{equation}\label{eqn:wt0upper}
\dim\left(\MPwt{0}{D}\right) \leq  \min_{\FB} r_{\FB}.
\end{equation}
Suppose that $\FB_0$ satisfies 
$$
r_{\FB_0} = \min_{\FB} r_{\FB}.
$$
By the inequality \eqref{eqn:wt0upper}, it suffices to construct $r_{\FB_0}$ linearly independent modular local constants with exceptional set $E_D$. 

Abbreviating $r:=r_{\FB_0}$, we denote the $r$ connected components which intersect $\FB_0$ by $\CC_1,\dots,\CC_r$.  We follow the construction of modular local polynomials in the proof of Proposition \ref{prop:modularP} (1).  For $n\in \left\{1,\dots, r\right\}$, we define $\PP_n$ such that 
$$
\PP_n(\tau):=\begin{cases}1 &\text{if }\tau \in \FB_0\cap \CC_{n},\\ 0 &\text{if }\tau\in \FB_0\cap \CC_j,\ j\neq n,\end{cases}
$$
and then (uniquely) specify the value of $\PP_n$ elsewhere so that (weight zero) modularity and \eqref{eqn:Favg} are satisfied.  These $r$ functions are clearly linearly independent.  By construction, $\PP_n$ is a modular local constant function with exceptional set 
$$
E_D\cup \bigcup_{\gamma\in \SL_2(\Z)} \gamma\left(\partial \FB_0\right).
$$
It remains to show that $\PP_n$ is a local constant function for the smaller exceptional set $E_D$.  Assume for contradiction that there is some point $\tau_0\in \partial\left(\FB_0\right)$ and $\tau_0\notin E_D$ for which $\PP_n$ has a discontinuity at $\tau_0$.  Since $\tau_0\notin E_D$, there exists a connected component $\CC$ of $\H\setminus E_D$ for which $\tau_0\in \CC$. 
Since $\CC$ is open, every sufficiently small neighborhood $\mathcal{N}$ around $\tau_0$ is contained in $\CC$.  For $\tau\in \mathcal{N}$, invariance under $\SL_2(\Z)$ of $\PP_n$ implies that 
\begin{equation}\label{eqn:Pnval}
\PP_n(\tau)=\begin{cases} 
1 &\text{if }\tau\in \CC_n\cap \FF_0 \pmod{\SL_2(\Z)},\\
0 &\text{if }\tau\in \CC_j\cap \FF_0 \pmod{\SL_2(\Z)}\text{ for }j\neq n.
\end{cases}
\end{equation}
Now note that if 
$$
\tau_0\notin \CC_j\pmod{\SL_2(\Z)}
$$
then the intersection of $\CC_j$ and $\CC$ modulo $\SL_2(\Z)$ is trivial (the intersection of any two connected components is either trivial or the entire connected component).  Hence in particular the intersection of $\mathcal{N}$ and $\CC_j$ modulo $\SL_2(\Z)$ is trivial.  

If $\tau_0\notin \CC_n\pmod{\SL_2(\Z)}$, then by \eqref{eqn:Pnval} we conclude that for every $\tau\in \mathcal{N}$ we have
$$
\PP_n(\tau)=0.
$$
But then $\PP_n$ is continuous at $\tau_0$, contradicting the assumption.  It follows that $\tau_0\notin \CC_n\pmod{\SL_2(\Z)}$.

Similarly, if $\tau_0\notin \CC_j\pmod{\SL_2(\Z)}$ for every $j\neq n$, then by \eqref{eqn:Pnval} we conclude that for every $\tau\in \mathcal{N}$ we have
$$
\PP_n(\tau)=1,
$$
which is again continuous at $\tau_0$.  We hence conclude that 
$$
\tau_0\in \CC_n\pmod{\SL_2(\Z)}
$$
and for some $j\neq n$
$$
\tau_0\in \CC_j\pmod{\SL_2(\Z)}.
$$
Without loss of generality, we may assume that $\tau_0\in \CC_n$ and choose $\gamma\in \SL_2(\Z)$ such that 
$$
\tau_0\in \gamma\left(\CC_j\right).
$$
Since $\tau_0\in \CC_n$, $\tau_0\in \gamma\left(\CC_j\right)$, and $\tau_0\in \CC$, we conclude that (since the intersection of any two connected components is either trivial or the entire connected component) 
\begin{equation}\label{eqn:CCnCCj}
\CC_n=\CC=\gamma\left(\CC_j\right).
\end{equation}
We now construct a new fundamental domain which contradicts the minimality of $r$, namely
$$
\FF_0:=\left\{ \tau\in \FB_0\big| \tau\notin \overline{\CC_{j}}\right\} \cup \gamma\left(\overline{\CC_j}\cap \FB_0\right),
$$
where $\overline{\CC_j}$ denotes the closure of $\CC_j$.  The set $\FF_0$ is a fundamental domain because each point in this set is equivalent to precisely one point in $\FB_0$ and $\FB_0$ is a fundamental domain.  Moreover, combining \eqref{eqn:CCnCCj} with the fact that 
$$
\FB_{0}\subseteq \bigcup_{\ell=1}^{r} \overline{\CC_{\ell}},
$$
we obtain that
$$
\FF_0\subseteq \bigcup_{\ell\neq j} \overline{\CC_{\ell}}.
$$
But then $r_{\FB_0} =r-1$, contradicting the minimality of $r$.  This contradiction implies that the function $\PP_n$ is indeed a modular local constant function with exceptional set $E_D$.  We hence obtain the desired lower bound on the dimension of this space, completing the proof.

\end{proof}

\section{The case $D=5$}\label{sec:ex}
We now return to the negative weight case and work out an example in detail.  
Using the relations built in the proof of Theorem \ref{thm:dim}, we determine the space of weight $-2$ modular local polynomials with exceptional set $E_5$.  In addition to the connected components $\CC_{\infty}$ and $\CC_{\rho}$ mentioned in the introduction, we require the connected component $\CC_{\rho-1}$ containing $\rho-1$ and for $\PP\in \MPwt{-2}{5}$ we denote the corresponding polynomials by $P_{\infty}$, $P_{\rho}$, and $P_{\rho-1}$, respectively.

Since $\PP\in \MPwt{-2}{5}$ is uniquely determined by its value inside each connected component intersecting the standard fundamental domain $\FF$, we first determine which connected components of $\H\setminus E_5$ intersect $\FF$.  For this, we need to compute the quadratic forms $Q=[a,b,c]\in \QQ_5$ for which $S_Q\cap \FF\neq\emptyset$.  Using the fact that $a\neq 0$, we may rewrite the formula determining $S_Q$ as
\begin{equation}\label{eqn:SQ}
a\left(\left|\tau+\frac{b}{2a}\right|^2-\frac{5}{4a^2}\right)=0.
\end{equation}
We conclude that $S_Q$ can only intersect $\FF$ if $a=\pm 1$, since otherwise every $\tau\in S_Q$ satisfies
$$
\im(\tau)\leq \left|\tau+\frac{b}{2a}\right|=\frac{\sqrt{5}}{2|a|}\leq \frac{\sqrt{5}}{4}<\frac{\sqrt{3}}{2},
$$
which contradicts $\tau\in \FF$.   Since $S_Q=S_{-Q}$, we may assume that $a=1$.  

Since $b^2\equiv 5\pmod{4}$, we have that $b$ is odd.  Therefore $b=\pm 1$ in the case that $S_Q\cap \FF\neq \emptyset$, since for $|b|\geq 3$ and $\tau\in \FF$ we have that
$$
\left|\tau+\frac{b}{2}\right|^2\geq \frac{3}{4} + \left|\frac{2\re(\tau)+ b}{2}\right|^2\geq \frac{3}{4}+ \frac{\left(|b|-1\right)^2}{4}\geq \frac{7}{4}>\frac{5}{4},
$$
contradicting $\tau\in S_Q$.  Therefore, the only semi-circles $S_Q$ which may intersect $\FF$ are those which correspond to 
\begin{align*}
Q_1&:=[1,1,-1],\\
Q_2&:=[1,-1,-1].
\end{align*}
Since $S_{Q_1}$ and $S_{Q_2}$ both clearly intersect $\FF$ and 
$$
S_{Q_1}\cap S_{Q_2}=\{ i \},
$$
we have precisely 3 connected components whose closure covers $\FF$, namely $\CC_{\infty}$, $\CC_{\rho}$, and $\CC_{\rho-1}$.  

However, we know from \cite{BKW} that $P_{\infty}$ is a constant $c_{\infty}$, while \eqref{eqn:Ptrans} states that
$$
P_{\rho-1}(X)=P_{\rho}(X+1),
$$
and \eqref{eqn:Prel} implies that
$$
P_{\rho}(X)=X^2 P_{\rho}\left(1-\frac{1}{X}\right).
$$
Hence, if $P_{\rho}(X)=aX^2+bX+c$ ($a,b,c\in \C$), then 
\begin{align*}
aX^2+bX+c &= X^2\left(a\left(1-\frac{1}{X}\right)^2 + b\left(1-\frac{1}{X}\right) + c\right)\\
& = \left(a+b+c\right)X^2 - \left(2a+b\right)X + a.
\end{align*}
It follows that $a=c=-b$, so that $P_{\rho}(X)=a\left(X^2-X+1\right)$. We conclude that $\dim\left(\MPwt{-2}{5}\right)\leq 2$.

In order to show the claim, it remains to show that $c_{\infty}$ and $a$ may be chosen arbitrarily, which is possible if and only if there is no point in $\CC_{\infty}$ which is $\SL_2(\Z)$-equivalent to a point in $\CC_{\rho}$. 

Assume for contradiction that there exists $\tau_0\in \CC_{\rho}$ and $\gamma=\left(\begin{smallmatrix}a&b\\c&d\end{smallmatrix}\right)\in \SL_2(\Z)$ such that $\gamma\tau_0\in \CC_{\infty}$.  Since $\CC_{\rho}$, $\CC_{\infty}$, and $\gamma\left(\CC_{\rho}\right)$ are open, there exists a neighborhood $\mathcal{N}\subseteq \CC_{\rho}$ of $\tau_0$ for which $\gamma\tau_0\in \CC_{\infty}$.  But then for every $\PP\in \MPwt{-2}{5}$ there exist $\alpha$ and $\beta$ such that for every $\tau\in \mathcal{N}$ we have
$$
\alpha\left(\tau^2-\tau+1\right)=\PP\left(\tau\right) = \left(c\tau+d\right)^2\PP\left(\frac{a\tau+b}{c\tau+d}\right) = \beta\left(c\tau+d\right)^2,
$$
which clearly implies that $\alpha=\beta=0$.  However, a weight $-2$ locally harmonic Maass form $\PP_0$ with exceptional set $E_5$ was constructed in \cite{BKW}.  Since there are no weight $4$ cusp forms on $\SL_2(\Z)$, it turns out from the theory in \cite{BKW} that $\PP_0$ is a modular local polynomial and $\beta\neq 0$, leading to a contradiction.  It follows that there are no points in $\CC_{\infty}$ which are equivalent to points in $\CC_{\rho}$, so the dimension is precisely 2.

\end{document}